\documentclass[a4paper,11pt]{amsart}
\usepackage[plainpages=false]{hyperref}
\usepackage{amsfonts,latexsym,rawfonts,amsmath,amssymb,amsthm, mathrsfs, lscape}
\usepackage{verbatim}

\usepackage[all]{xy}
\usepackage{graphicx,psfrag}

\usepackage{array,tabularx}

\usepackage{setspace}

\newtheorem{thm}{Theorem}
\newtheorem{cor}{Corollary}[section]
\newtheorem{lem}{Lemma}[section]

\newtheorem{prop}{Proposition}[section]

\theoremstyle{remark}
\newtheorem{rmk}{Remark}[section]

\theoremstyle{definition}

\numberwithin{equation}{section}

\def\p{\partial}
\def\R{\mathbb{R}}

\def\N{\mathbb{N}}

\def\a{\alpha}
\def\b{\beta}
\def\c{\gamma}

\def\D{\Delta}

\def \n {\nabla}

\begin{document}

\title[Harmonic heat flow on almost Hermitian manifolds]{The harmonic heat flow of almost complex structures}

\author{Weiyong He, Bo Li}

\address{Department of Mathematics, University of Oregon, Eugene, Oregon, 97403}
\email{whe@uoregon.edu}
\email{bol@uoregon.edu}

\begin{abstract}We define and study the harmonic heat flow for almost complex structures which are compatible with a Riemannian structure $(M, g)$. This is a tensor-valued version of harmonic map heat flow. We prove that if the initial almost complex structure $J$ has small energy (depending on the norm $|\nabla J|$), then the flow exists for all time and converges to a K\"ahler structure. We also prove that there is a finite time singularity if the initial energy is sufficiently small but there is no K\"ahler structure in the homotopy class. A main technical tool is a version of monotonicity formula, similar as in the theory of the harmonic map heat flow. 
We also construct an almost complex structure on a flat four tori with small energy such that the harmonic heat flow blows up at finite time with such an initial data.
\end{abstract}

\maketitle

\section{Introduction}

Almost complex manifolds contain well-studied objects in the modern theory of differential geometry, such as complex manifolds, symplectic manifolds and K\"ahler manifolds.  An almost complex structure supports compatible Riemannian metrics, called almost Hermitian structures. The study of general almost Hermitian manifolds has a rich history.

In this paper we study \emph{harmonic heat flow} on an almost Hermitian manifold $(M, g, J)$. Consider the metric $g$ is fixed and we look for a ``best" almost complex structure which are compatible with the metric. This problem dates back to Calabi-Gluck \cite{CG} and C. Wood \cite{Wood1, Wood2} in 1990s using the theory of twistor bundles. In particular, C. Wood \cite{Wood1} came up with the notion of the \emph{harmonic almost complex/Hermitian structure} by considering minimizing the energy functional, for all compatible almost complex structures, 
\begin{equation}
E(J)=\int_M |\nabla J|^2 dv.\end{equation}
The Euler-Lagrangian equation  reads
\[
[J, \Delta_g J]=0,
\]where $\Delta_g$ is the rough Laplacian of $g$. 

Clearly a K\"ahler structure $\nabla J=0$ gives an absolute minimizer of the energy functional. But there are  various absolute minimizers of the energy functional which are not K\"ahler, see for example \cite{BorSalvai}. In this sense one can view the energy-minimizing harmonic almost Hermitian structures as a natural generalization of K\"ahler structures. The harmonic almost complex structure has gained considerate interest and we refer the readers to the recent survey paper \cite{Davidov} for the background, history and results in this subject.

A straightforward computation shows that  the above Euler-Lagrangian equation is equivalent to the following equation \cite{He17},
\begin{equation}\label{E2}
\Delta_g J-J\nabla_p J\nabla_p J=0,
\end{equation}
where $J\nabla_p J\nabla_p J$ reads in local coordinates
\[(J\nabla_p J\nabla_p J)^k_j=g^{pq}J_j^a\nabla_pJ_a^b \nabla_q J_b^k.\]
Recently the first named author \cite{He17} has studied the regularity theory of weakly harmonic almost complex structures and has proved many results parallel to profound regularity theory of harmonic maps. The regularity theory of harmonic maps has a very rich history in differential geometry and has played a very central role in regularity theory of geometric analysis, with tremendous fascinating results and applications in literature. The harmonic map heat flow, first studied by Eells-Sampson \cite{ES} in 1960s, has been a very effective tool to construct the harmonic maps in a fixed homotopy class and has been studied extensively ever since.  
 
 In this paper we study the following \emph{harmonic heat flow} of an almost complex structure,
\begin{equation}\label{E3}
\p_t J=\Delta_g J-J\nabla_p J\nabla_p J
\end{equation}
This is a tensor-valued version of harmonic map heat equation. The short time existence follows from rather standard theory since \eqref{E3} is a parabolic system. We also derive Shi-type estimate of this equation, which shows that the flow can be extended once $|\nabla J|$ remains bounded. These results appear in Section 2. We study long time behavior and finite time singularities in Section 3. We summarize our main results as follows,

\begin{thm}\label{main1}Let $(M, g, J_0)$ be an almost Hermitian manifold such that $|\nabla J_0|\leq K$ for a positive constant $K$. Then there exists $\epsilon=\epsilon(K, g)$ such that if the energy $E(J_0)<\epsilon$, then the harmonic heat flow \eqref{E3} exists for all time and converges smoothly to a K\"ahler structure by subsequence. 
\end{thm}

We also derive a general theorem about finite time singularities,

\begin{thm}\label{main2}Let $(M, g, J_0)$ be an almost Hermitian manifold. Suppose  in the homotopy class of $[J_0]$ there exists no K\"ahler structure but \[\inf_{J\in [J_0]} E(J)=0.\] Then there exists an $\epsilon>0$ such that for any initial almost complex structure $J\in [J_0]$ with energy $E(J)<\epsilon$,  the harmonic heat flow \eqref{E3} develops a finite time singularity at $T<\infty$. In particular $T\rightarrow 0$ if $\epsilon\rightarrow 0$. 
\end{thm}

S. Donaldson \cite{Don} constructed a homotopy class of almost complex structures on a $K3$ surface which contains no complex structure in the homotopy class. Inspired by his example, we can construct an almost complex structure on a flat four-tori with arbitrary small energy but its homotopy class contains no K\"ahlerian complex structure which is compatible with the flat metric. 
 This produces an example of finite time singularity using Theorem \ref{main2}. 
These results are parallel to results of Chen-Ding \cite{CD}, built upon the work of Struwe \cite{S} and Chen-Struwe \cite{CS}. A major technical tool is  a version of monotonicity formula and  a version of $\epsilon$-regularity for the harmonic heat flow \eqref{E3}. We should emphasize that even though our methods are similar to those used for harmonic map heat flow, the tensor-valued version does require more careful analysis. In particular, the background geometry of $(M, g)$ gets involved in a significant way and it does pose extra difficulties that need to be taken care of. We can mention two examples. The first one is that we do not have a parallel theory as in Eells-Sampson \cite{ES}. Our ``target manifold" is fiber bundle modeled on $SO(2n)/U(n)$, which always have positive curvature as a symmetric space (the fiber space) and hence we do not have the parallel results as in \cite{ES}. Another example is that the lower order terms coming from curvature do have significant effects. 
In particular, our monotonicity formula behaves differently and it is more complicated than that in the harmonic map heat flow, due to its tensor-valued nature. This behavior requires extra care when we prove both Theorem \ref{main1} and Theorem \ref{main2}; compare \cite[Lemma 2.2]{CD} (due to Chen-Struwe \cite{CS}) and Theorem 3.2. 
Nevertheless, it is fair to say that we have a rather parallel theory  as in harmonic map heat flow, even though technically there are substantial differences. 
The theory of harmonic maps and harmonic map heat flow is a vast subject with hundreds (maybe thousands) of papers. We hope our study of harmonic heat flow for almost complex structures is just a start of a fruitful journey. \\

{\bf Acknowledgement:} The first named author is partly supported by an NSF grant, award no. 1611797. The second named author is supported in part by China Scholarship Council.

\numberwithin{equation}{section}
\numberwithin{thm}{section}

\section{The harmonic heat flow of an almost complex structure}
Let $(M, g, J_0)$ be an almost Hermitian structure. We consider the harmonic heat flow \eqref{E3} of an almost complex structure, with the initial condition $J(0)=J_0$. In this section, we prove the short time existence of the flow, and then derive some estimates along the flow.

\subsection{Short time existence}
\begin{thm}For any smooth initial $J_0$, there exists a unique smooth short time solution of \eqref{E3} such that $(g, J(t))$ defines a compatible almost Hermitian structures.
\end{thm}

\begin{proof}First we suppose $J_0$ is a smooth section of $\text{End}(TM)$ which does not have to be a compatible almost complex structure. Then the \eqref{E3} is a semilinear elliptic equation for $J_0$. The standard parabolic equation implies that there exists a unique smooth short time solution. Then we only need to argue that if $J_0$ is a compatible almost complex structure, then $J(t)$ remains to be a compatible almost complex structure. Namely we need to show that, along the flow, 
\[
J^2=-id, g(J, J)=g
\] 
Denote $A=g(J, J)-g$ and $B=J^2+id$.   We will show that on every  closed interval where the smooth solution $ J(t) $ exists, there is a constant $ c $ such that 
\begin{equation}
\begin{split}
\p_t |A|^2 \le \D |A|^2+c|A|^2 \\
\p_t |B|^2 \le \D |B|^2+c|B|^2
\end{split}
\end{equation}
Then since $ |A(0)|^2=|B(0)|^2=0 $, by maximal principle $ A $ and $ B $ are both zero along the flow. 

By $ \n B =J\n J +\n J J $, we compute 
\begin{equation}
\begin{split}
\p_t B & =J(\D J-J\n_p J \n_p J )+ (\D J-J\n_p J \n_p J )J\\
& = \D B -2 \n_p J \n_p J -J^2 \n_p J \n_p J- J\n_p J \n_p J J \\
&= \D B -2 B \n_p J \n_p J -J\n_p J \n_p B + J \n_p B \n_p J.
\end{split}
\end{equation}
 Then we have  \[ \p_t |B|^2 =\D |B|^2-2|\n B|^2+2\big( B,-2 B \n_p J \n_p J -J\n_p J \n_p B + J \n_p B \n_p J \big), \] 
On a closed interval, $ J $ and $ \n J $ are bounded, so that we have
\begin{equation}
\begin{split}
\p_t |B|^2 & \le \D |B|^2-2|\n B|^2 + c_1|B|(|B|+|\n B|)\\  
& \le \D |B|^2+c|B|^2.
\end{split}
\end{equation}
Hence $ B=0 $ follows.\\

Now we compute
\begin{equation}\label{A1}
\p_t A=\D A -2g(\n_p J,\n_p J)-g(J\n_p J \n_p J ,J)-g(J,J\n_p J \n_p J).
\end{equation}
Since $ \n A=g(\n J,J)+g(J,\n J) $, and  $ J\n J+\n J J=0 $ by $ B=0 $, we see that
\begin{equation}\label{A2}
\begin{split}
\n_p A(J\n_p J, \quad) &=g(\n_p J J \n_p J,J)+ g(J ^2 \n_p J, \n_p J)\\
& = -g(J\n_p J \n_p J ,J)-g(\n_p J,\n_p J).
\end{split}
\end{equation}
Similarly, 
\begin{equation}\label{A3}
\n_p A(\quad, J\n_p J)=-g(J,J\n_p J \n_p J )-g(\n_p J,\n_p J).
\end{equation}
So we can rewrite \eqref{A1} as
\begin{equation}
\p_t A =\D A + \n_p A(J\n_p J, \quad)+\n_p A(\quad, J\n_p J).
\end{equation}
On a closed interval, we then have
\begin{equation}
\begin{split}
\p_t |A|^2 & = \D |A|^2- 2 |\n A|^2+2\big( A, \n_p A(J\n_p J, \quad)+\n_p A(\quad, J\n_p J)   \big) \\
& \le  \D |A|^2- 2 |\n A|^2+ c_1 |A||\n A| \\
& \le  \D |A|^2 + c|A|^2.
\end{split}
\end{equation}
Hence $ A=0 $.

\end{proof}

\subsection{Evolution equation and Shi-type estimate}

The following Shi-type estimate on higher derivatives of $J$ holds. 
\begin{prop} Suppose that $K>0$ and the harmonic heat flow exists in $[0,  K^{-1}]$. For each $m\in \N_{\ge 2}$, there exists a constant $C_m$ depending only on $ K,n $ and the metric $ g $ such that if $|\nabla J|(x, t)\leq K$ on $M\times [0, K^{-1}]$, then for all $t\in [0,  K^{-1}]$, we have the estimate
\[
|\nabla^{m}J|\leq C_m t^{-(m-1)/2} . 
\]
\end{prop}

\begin{proof}
In this proof $ C $ represents a universal constant depending on $ K,n$, and $g$ which can vary line by line. 
First we compute $ \p_t \n^m J $:
\begin{displaymath}
\begin{split}
  \p_t \n^m J& = \n^m (\D J-J\n_p J \n_p J) \\
 & = \D \n^m J + \sum\limits_{i+j=m} \n^i Rm*\n^j J -\n^m (J\n_p J \n_p J) . 
\end{split}
\end{displaymath}
Then we compute $\p _t |\n^m J|^2$ as following: 
\begin{equation}\label{jm}
\begin{split}
\p_t |\n^m J|^2  =&\Delta |\n^m J|^2 -2| \n ^{m+1} J |^2 +  \sum\limits_{i+j=m}\n^i Rm *\n^j J* \n^m J\\
 & +\sum\limits_{i+j+k=m} \n^i J* \n^{j+1} J *\n^ {k+1}J*\n ^m J.
\end{split}
\end{equation}
Particularly, when $ m=1 $ we have 
\begin{equation}\label{eJ}
\p_t | \n J |^2  \le  \D | \n J|^2 +C\left( | \n J | + | \n J |^2 +  | \n J |^4 \right).
\end{equation}
 And for $ m=2 $ we have 
\begin{displaymath}
\begin{split}
\p_t | \n ^2 J |^2  \le \D | \n ^2 J |^2 +C( | \n ^2 J |^3+| \n ^2 J |^2+| \n ^2 J | )  .
\end{split}
\end{displaymath}
The above inequality is equivalent to  \[ 2| \n ^2 J | \p_t | \n ^2 J | \le 2| \n ^2 J | \D | \n ^2 J | + 2| \n | \n ^2 J | |^2 + C( | \n ^2 J |^3+| \n ^2 J |^2+| \n ^2 J | ), \]
 or
 \begin{equation}\label{jm2}
 \p_t | \n ^2 J |  \le \D | \n ^2 J | + \frac{| \n | \n ^2 J | |^2}{| \n ^2 J |} + C(| \n ^2 J |^2+| \n ^2 J |+1).
 \end{equation}

To prove the proposition for $ m=2 $, let \[ f=t^{3/2}| \n ^2 J|+\b t|\n J|^2 -C_2t ,\quad t\in [0,K^{-1}],  \] where $ \b  $ and $ C_2=C_2(K,n,g) $ are constants to be determined later. Then the inequality \eqref{jm2} leads to  
\begin{equation}\label{f1}
\begin{split}
(\p_t- \D)  f \le& \frac{3}{2} t^{1/2} | \n ^2 J|+\b |\n J|^2 -C_2 \\
&+ t^{3/2}   \left(\frac{| \n | \n ^2 J |  |^2}{| \n ^2 J |} + C(| \n ^2 J |^2+| \n ^2 J |+1)\right)\\
&+t\b (-| \n^2 J |^2 + C).
\end{split}
\end{equation}
If $ f $ reaches the maximum at some point $(x,t)$ with $ t>0 $, we have $$ 0=\n f=t^{3/2} \n | \n^2 J |+2t\b \left(  \n^2 J, \n J \right)   $$
So that at $ (x, t) $, $ \D f \le 0 $ and 
\begin{equation}
t| \n | \n ^2 J | |^2 \le \b^2  C| \n ^2 J |^2.
\end{equation}
Apply this to \eqref{f1}, with the following two estimates:
 \begin{displaymath}
 \begin{split}
  ( \frac{3}{2}+\b ^2 C+ tC )t^{1/2}| \n ^2 J |& \le C(\b)+t| \n^2 J |^2, \\
   t^{3/2}C| \n ^2 J |^2 & \le Ct| \n ^2 J |^2,
 \end{split}
 \end{displaymath}
where $ C(\b) $ is a constant depending on $ \b $, we have  
  \begin{equation}
  \p_t f \le t(C-\beta)| \n ^2 J |^2+ C(\beta)-C_2.
  \end{equation}
We may choose $ \b > C $  and then choose $ C_2 \ge C(\beta)  $, so that 
 \[  \p_t f \le 0 \] 
at the maximum point.
 By maximum principle we know that $ f \le 0 $, i.e.  $$ |\n^2 J| \le C_2 t^{-1/2}. $$ \\
For $ m \ge 3 $ we prove by induction. Assume that we have estimated 
\begin{displaymath}
|\n^k J| \le C_k t^{-(k-1)/2}, \quad 2 \le k < m.
\end{displaymath}
  Note that  $ t^{-\frac{j}{2}} \le Ct^{-\frac{k}{2}} $ when $ j \le k $ since $ t \le K^{-1} $ is bounded, by \eqref{jm} and the inductive assumption  we have, for $2 \le k <m $,
\begin{equation}\label{jm3}
\begin{split}
 (\p_t -\D) | \n^k J |^2  & \le -2|\n^{k+1} J|^2 +Ct^{-(2k-1)/2}+Ct^{-(k-1)/2}|\n^{k+1} J|\\
  & \le -|\n^{k+1} J|^2 +Ct^{-k} .
\end{split}
\end{equation}
And for $ m $, we have 
\begin{equation}\label{jmm}
\begin{split}
(\p_t -\D) | \n^m J |^2   \le & -2|\n^{m+1} J|^2 +Ct^{-(m-1)/2}|\n^m J|\\ 
&+C |\n^m J|(|\n^{m+1} J|+t^{-1/2}|\n^m J|+t^{-m/2})\\
   \le & Ct^{-1/2}|\n^m J|^2  +Ct^{-m}.
\end{split}
\end{equation}
Let \[ G(t)= t^m| \n^m J |^2 + \sum\limits_{k=2}^{m-1} \frac{\b}{k!}t^k | \n^k J |^2, \quad t\in [0,K^{-1}]. \] 
where $ \b $ is to be determined. Then by \eqref{jm3} and \eqref{jmm} we compute 
\begin{displaymath}
\begin{split}
(\p_t-\D)G \le & mt^{m-1}| \n^m J |^2+ \sum\limits_{k=2}^{m-1} \frac{\b}{(k-1)!}t^{k-1} | \n^k J |^2 \\
& +Ct^{m}( t^{-\frac{1}{2}}| \n^m J |^2+t^{-m}) + \sum\limits_{k=2}^{m-1} \frac{\b}{k!}t^{k} ( -| \n^{k+1} J |^2+Ct^{-k} ) \\
\le & Ct^{m-1}|\n^m J|^2-\frac{\b}{(m-1)!}t^{m-1}|\n^m J|^2+\b t|\n^2 J|^2+C(\b).
\end{split}
\end{displaymath}
Choose $ \b $ large enough so that \[ (\p_t-\D)G \le \b C_2^2 +C(\b). \] Let $ C_m^2=\b C_2^2 +C(\b) $. Then by maximum principle we have 
\[ G(t)\le G(0)+C_m^2 t =C_m^2 t. \]
  Hence $ t^m| \n^m J |^2 \le G \le C_m^2 t $, i.e.
$$ | \n^m J | \le C_m t^{-(m-1)/2}. $$
\end{proof}

\section{Long time existence and finite time singularity}

We prove Theorem \ref{main1} and Theorem \ref{main2} in this section. First we derive some monotonicity formulas and then we prove an $\epsilon$-regularity theorem. We also prove that a harmonic almost complex structure with small energy has to be a K\"ahler structure. Technically a major difference with the harmonic map heat flow lies in the monotonicity formula. There are  terms involved with the curvatures which have different orders (compared with the harmonic map heat flow). Hence the estimates to derive the monotonicity formula are quite different, in particular its dependence of the initial energy $E_0$, see \eqref{z1} and \eqref{z2}. 

\subsection{Monotonicity formula}\label{sub31}

We shall derive some monotonicity formula along the harmonic heat flow for almost complex structures. 
Similar monotonicity formulas using backward heat kernel are well-known for mean curvature flow and harmonic map heat flow, see Struwe \cite{S}, Huisken \cite{H} and Hamilton \cite{Hamilton82}.

First we derive a version of Hamilton's type monotonicity formula. Suppose $J(t)$ is a smooth solution exists on $[0, T)$ for $T<\infty$, the maximal existence time.  Let $k$ be a positive solution of the backward heat equation on $M$,
\[
\p_t k+\Delta k=0, \int_M k=1
\]
We consider the quantity 
\[
Z(t)=(T-t)\int_M |\nabla J|^2 k .
\]
\begin{lem}We have the following,
\begin{equation}\label{m0}
\begin{split}
\p_t Z&+2(T-t)\int_M \left|\Delta J-J\nabla_pJ \nabla_p J+\frac{\nabla_i k\nabla_i J}{k}\right|^2 k\\
&+2(T-t)\int_M \left(\nabla_i\nabla_j k-\frac{\nabla_i k\nabla_j k}{k}+\frac{kg_{ij}}{2(T-t)}\right)\left(\nabla_i J, \nabla_j J\right)\\
=&2(T-t)\int_M \nabla_i k\left(\nabla_j J, [\nabla_i, \nabla_j] J\right).
\end{split}
\end{equation}
\end{lem}
\begin{proof}
We compute
\begin{equation}\label{m1}
\begin{split}
\p_t Z=&-\int_M |\nabla J|^2 k +(T-t)\int_M \p_t (|\nabla J|^2 k)\\
=&-\int_M |\nabla J|^2 k+(T-t)\int_M 2\left(\nabla J, \nabla (\p_t J)\right)k-\Delta k |\nabla J|^2\\
=&-\int_M |\nabla J|^2 k-2(T-t)\int_M \left(k\Delta J+\nabla_i k \nabla_i J, \p_t J\right)\\
&-(T-t)\int_M \Delta k |\nabla J|^2
\end{split}
\end{equation}
We compute
\begin{equation}\label{m2}
\begin{split}
\int_M \left(k\Delta J+\nabla_i k \nabla_i J, \p_t J\right)=&\int_M \left(k\Delta J+\nabla_i k\nabla_i J, \Delta J-J\nabla_p J\nabla_p J\right)
\end{split}
\end{equation}
We compute that 
\begin{equation}\label{m3}
\begin{split}
\int_M \left(\nabla_i (k\nabla_i J), J\nabla_p J\nabla_p J\right)=&\int_M \left(\nabla_i (k\nabla_i J_j^k), J_j^a\nabla_pJ^b_a \nabla_p J^k_b\right)\\
=&-\int_M k\left(\nabla_i J^k_j, \nabla_i (J_j^a\nabla_pJ^b_a \nabla_p J^k_b)\right)\\
=&\int_M k \left(\nabla_i J^k_j \nabla_i J^a_j, \nabla_pJ^a_b\nabla_pJ^k_b\right)
\end{split}
\end{equation}
where we use the following pointwise identity, 
\[
\left(\nabla_i J^k_j,  J_j^a(\nabla_i\nabla_pJ^b_a) \nabla_p J^k_b\right)+\left(\nabla_i J^k_j,  J_j^a\nabla_pJ^b_a (\nabla_i\nabla_p J^k_b)\right)=0
\]
Hence by \eqref{m3}, we have
\begin{equation}\label{m4}
\begin{split}
\int_M \left(\nabla_i (k\nabla_i J), J\nabla_p J\nabla_p J\right)=&\int_M k \left(\nabla_i J^k_j \nabla_i J^j_a, \nabla_pJ^k_b\nabla_pJ^b_a\right)\\
=&\int_M k |\nabla_i J \nabla_i J|^2.
\end{split}
\end{equation}
Denote 
\[
I:=\int_M \left|\Delta J-J\nabla_pJ \nabla_p J+\frac{\nabla_i k\nabla_i J}{k}\right|^2 k
\]
Using \eqref{m2} and \eqref{m4}, it is straightforward to check that we have, 
\begin{equation}\label{m5}
\begin{split}
I=\int_M \left(k\Delta J+\nabla_i k\nabla_i J, \p_t J\right)+\int_M \frac{|\nabla_i k \nabla_i J|^2}{k}+\int_M (\nabla_i k\nabla_i J, \Delta J)
\end{split}
\end{equation}
We compute 
\begin{equation}
\begin{split}
\int_M \frac{|\nabla_i k \nabla_i J|^2}{k}=&-\int_M \left(\nabla_i\nabla_j k-\frac{\nabla_i k \nabla_j k}{k}\right)(\nabla_i J, \nabla_j J)\\
&-\int_M \nabla_j k(\Delta J, \nabla_j J)-\int_M \nabla_j k (\nabla_{i}J, \nabla_{i}\nabla_j J)
\end{split}
\end{equation}
Denote 
\[
II:=\int_M \left(\nabla_i\nabla_j k-\frac{\nabla_i k\nabla_j k}{k}+\frac{kg_{ij}}{2(T-t)}\right)\left(\nabla_i J, \nabla_j J\right)
\]
Using \eqref{m1}, \eqref{m4} and \eqref{m5}, we compute
\begin{equation}\label{m6}
\p_t Z+2(T-t)(I+II)=2(T-t)\int_M \nabla_i k\left(\nabla_j J, [\nabla_i, \nabla_j] J\right)
\end{equation}
This completes the proof. 
\end{proof}

\begin{rmk}The term on the righthand side in \eqref{m0} is zero in the harmonic map heat flow. This term needs extra care in the estimates. Similar terms will also appear in the following when we consider a local version. 
\end{rmk}

To make use of the formula derived above, one would need to estimate the backward heat kernel, and in particular the quantity $\nabla_i\nabla_j k-\frac{\nabla_i k\nabla_j k}{k}+\frac{kg_{ij}}{2(T-t)}$ locally, as in Hamilton's paper \cite{Hamilton821}. Note that this term is zero when the metric is Euclidean. 

The following local version of monotonicity formula is more relevant for our purpose, which is similar as in the harmonic map flow developed in \cite{S} and \cite{CS}. Without  loss of generality we assume  the injectivity radius  of $(M, g)$ is greater than $1$. Then  at an arbitrary point $ y\in M $, let $ x_i $ be a normal coordinate centre at $ y $, via which the geodesic ball $ B_1(y) \subset M $ is diffeomorphic to the Euclidean ball $ B_1(0) \subset \R^n $.  We can regard $ J $ as a tensor defined on $ B_1(0) \times [0,T) \subset \R^n \times [0,T)$. For any $ 0<T_0 \le T $, define 
\begin{equation}\label{Zt}
\begin{split}
& Z(t)=(T_0-t) \int_{\R^n} |\n J|^2 G \phi^2 \sqrt{|g|} dx, \\
& \Psi(R) =\int_{T_0-4R^2}^{T_0-R^2} \int_{\R^n} |\n J|^2 G \phi^2 \sqrt{|g|} dx dt,
\end{split}
\end{equation} 
where $ G=\frac{1}{(4\pi (T_0-t))^{n/2}} \exp(-\frac{|x|^2}{4(T_0-t)}) $ is the Euclidean heat kernel and $ \phi $ is a test function whose support is contained in  $ B_1(0) $, plus that $ \phi=1 $ in $ B_{1/2}(0) $. The norm $ |\n J|^2 $ is taken in the metric $ g $, i.e.
 \[ |\n J|^2=g_{i_1i_2}g^{j_1j_2} g^{k_1k_2}\n_{k_1}J_{j_1}^{i_1}\n_{k_2}J_{j_2}^{i_2} , \] 
where
 \[ \n_k J_\a^\b= \frac{\partial}{\partial x_k}J_\a^\b +\Gamma_{lk}^\b J_\a^l - \Gamma_{\a k}^l J_l^\b .\]

\begin{thm}\label{mon1}
For any $ T_0-\min \{ T_0,1 \}<t_1\le t_2 <T_0 $ and $ N>1 $ there holds the monotonicity formula
\begin{equation}\label{z1}
 Z(t_2)\le e^{C(f_2-f_1)}Z(t_1)+C\left( N^{n/2} (E_0+\sqrt{E_0})+\frac{1}{\ln ^2 N} \right)(t_2-t_1) 
\end{equation}
with a uniform constant $ C $ depending only on $ (M,g) $,
where \[ f(t)=-(T_0-t)\ln^2(T_0-t)+2(T_0-t)\ln (T_0-t)-3(T_0-t) .\]
\end{thm}

\begin{thm}\label{mon2}
For any $ 0 < R_2 \le R_1 \le \min \{\sqrt{T_0}/2,1 \} $ and $ N>1 $ there holds the monotonicity formula
 \begin{equation}\label{z2} 
 \Psi(R_2)\le e^{C(\tilde{f}_2-\tilde{f}_1)}\Psi(R_1)+C\left( N^{n/2} (E_0+\sqrt{E_0})+\frac{1}{\ln ^2 N} \right) ( R_1-R_2 ) 
  \end{equation}
   with a uniform constant $ C $ depending only on $ (M,g) $, where \[ \tilde{f}(R)=-4R^2\ln ^2 R +4R^2\ln R -3R^2 . \] 
\end{thm}

Before we treat the monotonicity formulas, first we prove two identities that will be used in the further computation.
\begin{lem}\label{lemma1}There holds
 \begin{equation}
 \begin{split}
 &\left(  \n_i J, J\n_p J\n_p J  \right)=0 \\
 &\left(  \D J-J\n_p J\n_p J, J\n_p J\n_p J  \right)=0
 \end{split}
 \end{equation}
\end{lem}

\begin{proof}
For arbitrary point we take a normal coordinate so that $ J^2=1 $ and $ g(J,J)=g $ are equivalent to 
\begin{equation} 
J_\b^\c J_\a^\b =-\delta_\a^\c  , \quad  J_\a^\b=-J_\b^\a 
\end{equation}  
in local coordinate. And $ J\n J+\n J J=0  $ implies \begin{equation}  
J_\a^\b \n J_\b^\c =-\n J_\a^\b  J_\b^\c. 
\end{equation} 
 Then we have
\begin{displaymath}
\begin{split}
\left( \n_i J, J\n_p J\n_p J \right) &=\n_i J_\a^\b  J_j^\b \n_p J_k^j \n_p J_\a^k \\
& = -  \n_i J_j^\b J_\a^\b (-\n_p J_j^k) (-\n_p J_k^\a) \\
&=-\left( \n_i J, J\n_p J\n_p J \right).
\end{split}
\end{displaymath}
This proves  $ \left(  \n_i J, J\n_p J\n_p J  \right)=0 $. \\

For the second equation of the lemma we use
 \[ J\D J+\D J J +2\n_p J \n_p J=0 ,\] so that 
  \begin{equation}
  \begin{split}
 4 \left| \n_p J \n_p J \right|^2 &= \left| J\D J+\D J J \right|^2 \\
 &= \left|  J\D J \right|^2 +\left|  \D J J \right|^2+2\left(  J\D J, \D J J\right).
  \end{split}
  \end{equation}
  Note that \[ \left|  \D J J \right|^2= \D J_\a^j J_j^\b \D J_\a^k J_k^\b = \D J_\a^j \D J_\a^k \delta_{jk} =\left| \D J \right|^2 \]
  and \[ \left(  J\D J, \D J J\right) =\left(  J\D J, -J \D J -2\n_p J \n_p J\right)=-\left| \D J \right|^2-2\left( J \D J, \n_p J \n_p J \right) , \]
  we have 
  \begin{equation}
  \left| \n_p J \n_p J \right|^2 = - \left( J \D J, \n_p J \n_p J \right) .
  \end{equation} 
  This implies the result.
\end{proof}

\begin{proof}[Proof of Theorem \ref{mon1}]

Compute 
\begin{displaymath}
\begin{split}
\frac{d}{dt} Z(t)  = & -\int_{\R^n} |\n J|^2 G \phi^2 \sqrt{|g|} dx \\
& -2(T_0-t) \int \left| \D J  - J\n_p J \n_p J- \frac{x\cdot \n J }{2(T_0-t)} \right|^2  G \phi^2 \sqrt{|g|} dx  \\ 
&-2(T_0-t)\int \left(-\frac{|x\cdot \n J|^2}{4(T_0-t)^2}G \phi^2 \sqrt{|g|}+\frac{x\cdot \n J \cdot \D J}{2(T_0-t)}G \phi^2 \sqrt{|g|} \right) dx  \\
&- 2(T_0-t)\int \n J\cdot (\D J-J\n_p J \n_p J)G(2\phi \n \phi \sqrt{|g|} + \frac{1}{2}\phi^2 g^{ij}\n g_{ij} \sqrt{|g|} ) dx \\
&+(T_0-t)\int_{\R^n} |\n J|^2 \p_t G \phi^2 \sqrt{|g|} dx \\
=& I+II+III+IV+V.
\end{split}
\end{displaymath}
We estimate $ IV $ as following,
\begin{displaymath}
\begin{split}
|IV|  \le & 2(T_0-t)\int \left| \n J\cdot (\D J-J\n_p J \n_p J +\frac{x\cdot \n J }{2(T_0-t)})G(2\phi \n \phi \sqrt{|g|} + \frac{1}{2}\phi^2 g^{ij}\n g_{ij} \sqrt{|g|} ) \right| dx\\
&+2(T_0-t)\int \left| \n J\cdot (x \cdot \n J)G(2\phi \n \phi \sqrt{|g|} + \frac{1}{2}\phi^2 g^{ij}\n g_{ij} \sqrt{|g|} ) \right| dx\\
\le & \frac{1}{2} |II| + 2C(T_0-t)\int |\n J|^2 |\n \phi |^2  G \sqrt{|g|}dx +C Z(t) \\
&+ CZ(t)+C\int |\n J|^2 G(\phi-\phi^2)\sqrt{|g|}dx.
\end{split}
\end{displaymath}
If $ T_0-t>\frac{1}{N} $, we have $ G<CN^{n/2} $, $ (T_0-t)G<CN^{n/2-1}<CN^{n/2} $. So that \[ |IV|<\frac{1}{2} |II|+CZ(t)+CN^{n/2}E_0. \]
If $ T_0-t \le \frac{1}{N}<1 $, note that in $ B_{1/2}(0) $, $ \n \phi=\phi-\phi^2=0 $, and  $ G $ is bounded outside $ B_{1/2}(0) $, so 
\[ |IV|  \le  \frac{1}{2} |II| + CZ(t)+CE_0 .\]
So in general for $ 0<t<T_0 $ there is a constant $C$ such that  
\begin{equation}\label{IV}
 |IV|<\frac{1}{2} |II|+CZ(t)+CN^{n/2}E_0.
  \end{equation}
To estimate $ I+III+V $, we compute 
\begin{displaymath}
\begin{split}
III =&\frac{1}{2(T_0-t)}\int |x \cdot \n J|^2 G \phi^2 \sqrt{|g|}dx - \int x\cdot \n J \cdot \D J  G \phi^2 \sqrt{|g|}dx \\
=& \frac{1}{2(T_0-t)}\int |x \cdot \n J|^2 G \phi^2 \sqrt{|g|}dx  +\int \n (x \cdot \n J) \n J G \phi^2 \sqrt{|g|}dx \\
& -\int (x \cdot \n J) \frac{\n J\cdot xG}{2(T_0-t)}  \phi^2 \sqrt{|g|}dx + \int (x\cdot \n J)( \n J \n \phi )2\phi G\sqrt{|g|}dx\\
=&   \int \n_j x_i \n_i J \n_j JG \phi^2 \sqrt{|g|}dx
+ \int \frac{1}{2}x_i \n_i |\n J|^2 G \phi^2 \sqrt{|g|}dx\\
& - \int x_i [\n_i,\n_j]J \n_j J G \phi^2 \sqrt{|g|}dx
+ \int (x\cdot \n J)( \n J \n \phi )2\phi G\sqrt{|g|}dx \\
=&  \int \n_j x_i \n_i J \n_j JG \phi^2 \sqrt{|g|}dx -\int \frac{g^{ij} \sigma_{ij}}{2} |\n J|^2 G \phi^2 \sqrt{|g|}\\
& + \frac{1}{4(T_0-t)}\int |x|^2 |\n J|^2 G \phi^2 \sqrt{|g|}dx \\& - \int x_i [\n_i,\n_j]J \n_j J G \phi^2 \sqrt{|g|}dx
+ \int (x\cdot \n J)( \n J \n \phi )2\phi G\sqrt{|g|}dx.
\end{split}
\end{displaymath}
So that
\begin{displaymath}
\begin{split}
I+III+V= & \int (\sigma^{ij}-g^{ij})\n_i J \n_j J G \phi^2 \sqrt{|g|}dx\\
&+\int  \frac{n-trace(g)}{2}|\n J|^2 G \phi^2 \sqrt{|g|}dx \\
& - \int x_i [\n_i,\n_j]J \n_j J G \phi^2 \sqrt{|g|}dx\\
&+ \int (x\cdot \n J)( \n J \n \phi )2\phi G\sqrt{|g|}dx .
\end{split}
\end{displaymath}
Similar to the estimate in $ IV $, we have 
\[  \left|  \int (x\cdot \n J)( \n J \n \phi )2\phi G\sqrt{|g|}dx \right| \le CE_0 . \]
Since $ |\sigma^{ij}-g^{ij}| \le C|x|^2 $ and $ |n-trace(g)| \le C|x|^2 $, we have \[ | I+III+V | \le CE_0+C\int |x|^2|\n J|^2 G \phi^2 \sqrt{|g|}dx+ C\int |x||\n J| G \phi^2 \sqrt{|g|}dx .\]

We now estimate
\begin{displaymath}
\begin{split}
\int |x|^2|\n J|^2 G \phi^2 \sqrt{|g|}dx  \le & \int_{|x|^2 \le (T_0-t)\ln^2(T_0-t)} |x|^2|\n J|^2 G \phi^2 \sqrt{|g|}dx \\
 &+ \int_{|x|^2 > (T_0-t)\ln^2(T_0-t)} |x|^2|\n J|^2 G \phi^2 \sqrt{|g|}dx \\
 \le & \ln^2(T_0-t)Z(t)+C\frac{\exp\{-\ln^2(T_0-t)/4\}}{(T_0-t)^{n/2}}E_0 \\
 \le & \ln^2(T_0-t)Z(t)+CE_0
\end{split}
\end{displaymath}
 since  $ \frac{\exp\{-\ln^2(T_0-t)/4\}}{(T_0-t)^{n/2}}  $ is uniformly bounded for $ 0 < t <T_0 $.
 
  As for the term $\int |x||\n J| G \phi^2 \sqrt{|g|}dx $, if $ T_0-t>\frac{1}{N} $, by $ G<CN^{n/2} $ we have \[ \int |x||\n J| G \phi^2 \sqrt{|g|}dx  \le CN^{n/2}\sqrt{E_0}. \] 
If $ T_0-t\le \frac{1}{N}<1 $, we have 
\begin{displaymath}
\begin{split}
 \int |x||\n J| G \phi^2 \sqrt{|g|}dx &\le \dfrac{1}{4\ln ^2 (T_0-t)}\int \dfrac{|x|^2}{T_0-t} G\phi^2 \sqrt{|g|}dx + \ln ^2(T_0-t) Z(t) \\
 & \le \dfrac{C}{\ln ^2 N}+\ln ^2(T_0-t) Z(t).
\end{split}
\end{displaymath} \\
 To sum up, we have, for $ 0<t<T_0 $,
 \begin{equation}\label{V}
 |I+III+V|  \le C \ln^2(T_0-t)Z(t)+CE_0+ \dfrac{C}{\ln ^2 N} + CN^{n/2}\sqrt{E_0}
 \end{equation}
 
 Inequalities \eqref{IV} and \eqref{V} give us
 \[ \frac{d}{dt}Z(t) \le -\frac{1}{2}|II|+ C\ln^2(T_0-t)Z(t)+CZ(t)+CN^{n/2}(E_0+\sqrt{E_0})+\frac{C}{\ln ^2 N} \] for some constant $ C $.
Let $ f(t)=-(T_0-t)\ln^2(T_0-t)+2(T_0-t)\ln (T_0-t)-3(T_0-t) $ so that $ f'(t)=\ln^2(T_0-t)+1 $. Then 
\begin{displaymath}
\begin{split}
\frac{d}{dt}\left( e^{-Cf}Z \right)= & e^{-Cf}\left( \frac{dZ}{dt} -C\ln^2(T_0-t)Z-CZ \right) \\
\le  & Ce^{-Cf}\left( N^{n/2}(E_0+\sqrt{E_0})+\frac{1}{\ln ^2 N}  \right),
\end{split}
\end{displaymath}
and the result follows.

\end{proof}

\begin{proof}[Proof of Theorem \ref{mon2}]
  
Let $ \alpha=\frac{R_2^2}{R_1^2} \le 1 $ and $ t=\alpha \tilde{t}+(1-\a) T_0 \ge \tilde{t} $. By Theorem \ref{mon1} we have
\begin{displaymath}
\begin{split}
\Psi(R_2)&=\int_{T_0-4R_2^2}^{T_0-R_2^2}\dfrac{Z(t)}{T_0-t}dt \\
&=\int_{T_0-4R_1^2}^{T_0-R_1^2}\dfrac{Z(t)}{T_0-\tilde{t}}d\tilde{t}\\
&\le \int_{T_0-4R_1^2}^{T_0-R_1^2} e^{C(f(t)-f(\tilde{t}))} \dfrac{Z(\tilde{t})}{T_0-\tilde{t}}+C\left( N^{n/2}(E_0+\sqrt{E_0})+\frac{1}{\ln ^2 N}  \right)\dfrac{t-\tilde{t}}{T_0-\tilde{t}}d\tilde{t}\\
&\le e^{ C(\tilde{f}(R_2)-\tilde{f}(R_1))  } \Psi(R_1)+C\left( N^{n/2}(E_0+\sqrt{E_0})+\frac{1}{\ln ^2 N}  \right)(R_1-R_2).
\end{split}
\end{displaymath}
Here we use the fact that the function $ f(\a\tilde{t} +(1-\a)T_0)-f(\tilde{t}) $ is decreasing. 
\end{proof}

\subsection{$ \epsilon $ regularity }\label{sub32}

 Denote $ P_r(x_0,t_0) = \{ (x,t)| d(x, x_0)\le r , |t-t_0|\le r^2  \} $ and $ P_r=P_r(0,T_0) $. The monotonicity allows us to prove the following $ \epsilon $ regularity theorem:

 \begin{thm}\label{eps}
 Let $ \Psi $ be defined as above with $ G=G_{(0,T_0)}(x,t) $ the heat kernel of $ \R^n $. There exists a constant $0< \epsilon_0 < \sqrt{T_0}/2 $  such that for a solution $ J $ in $ B_1(y) \times [0,T] $ with $ E_0 < \infty $, the following is true:
 
   If for some $ R \in (0, \epsilon_0) $ there holds \[ \Psi(R)<\epsilon_0, \] then \[ \sup\limits_{P_{\sigma R} } |\n J|^2 \le c(\sigma R)^{-2}  \] with some constant $ \sigma $ and $ c $ depending only on $  M$ .

 \end{thm}
  
  \begin{proof}

   Recall that if $ J $ is a solution, denoting $ e(J)=\frac{1}{2} |\n J|^2 $, then by \eqref{eJ}  we have  
   \begin{equation}\label{eJ1}
   (\p_t -\D )e(J) \le C(e(J)^2+1 )
   \end{equation}  for some constant $C$. 
  
    For $ r=\sigma R>0 $ with $ \sigma <1 $ to be determined later, let  $ \rho \in [0,r] $ and $ z_0=(x_0,t_0) \in P_\rho $ be such that 
    \begin{equation}
    \begin{split}
    &(r-\rho)^2\sup\limits_{P_\rho}e(J)=\max\limits_{0\le \a \le r} \{ (r-\a)^2\sup\limits_{P_\a}e(J) \}, \\
    & e(J)(z_0)=\sup\limits_{P_\rho } e(J)=e_0.
    \end{split}
    \end{equation}
    Set $ \rho_0=\frac{1}{2}(r-\rho) $. By the choice of $ \rho $ and $ z_0 $, $$ \sup\limits_{P_{\rho_0(x_0,t_0)}} e(J) \le \sup\limits_{P_{\rho_0+\rho}} e(J) \le 4e_0. $$
    Now define $ \tilde{J}(x,t) = J(x_0+\frac{x}{\sqrt{e_0}},t_0+\frac{t}{e_0}) $ in $ P_{r_0} $, where $ r_0=\sqrt{e_0} \cdot \rho_0 $, then $ e(\tilde{J})(0,0)=1 $ and $ \sup\limits_{P_{r_0}} e(\tilde{J}) \le 4$. $ \tilde{J} $ is a solution in $ P_{r_0}  $ with respect to the metric $ \tilde{g}(x)=g(x/\sqrt{e_0}) $.  We have 
    \begin{equation}
    (\p_t -\tilde{\D} )e(\tilde{J}) =\frac{1}{e_0^2} (\p_t -\D )e(J) \le C(e(\tilde{J})+\frac{1}{e_0^2}).
    \end{equation}
     If $ r_0>1 $, Harnack inequality implies 
     \begin{equation}
     1+\frac{1}{e_0^2}=e(\tilde{J})(0,0)+\frac{1}{e_0^2} \le C \int_{P_1} (e(\tilde{J})+\frac{1}{e_0^2})dxdt.
     \end{equation}
 Let $ \delta =1/\sqrt{e_0} \le \rho_0 $ then the scaling back
     \begin{displaymath}
     \int_{P_1} e(\tilde{J})dxdt= \delta ^{-n} \int_{P_\delta(x_0,t_0)} e(J)dxdt 
     \end{displaymath}
 Choose $ \sigma<1 $ so that $ r<R $. Since $ G_{(x_0,t_0+2\delta^2)}(x,t)\ge C\delta^{-n} $ and $ \delta+ \rho \le \rho_0+\rho \le r $,  the monotonicity formular \eqref{z2} implies 
 \begin{equation}
     \begin{split}
     \delta ^{-n} & \int_{P_\delta(x_0,t_0)} e(J)  dxdt \le C \int_{P_\delta(x_0,t_0)} e(J)G_{(x_0,t_0+2\delta^2)}(x,t) \phi^2 \sqrt{|g|}dxdt \\
     & \le C\int_{T_0-4R^2}^{T_0-R^2} \int_{\R^n} e(J)G_{(x_0,t_0+2\delta^2)}(x,t)\phi^2 \sqrt{|g|} dxdt +CR(E_0+1)
     \end{split}
     \end{equation}
 Now on $ T_0-4R^2<t<T_0-R^2 $, for given $ \epsilon >0 $, if $ \sigma $ is small enough: 
     \begin{equation}
         \begin{split}
        G_{(x_0,t_0+2\delta^2)}(x,t) &\le \frac{C}{(4\pi |t-T_0|)^{n/2}}\exp \left(-\frac{|x-x_0|^2}{4(t_0+2\delta^2-t)} \right) \\
         & \le C\exp \left(\frac{|x|^2}{4|t-T_0|}-\frac{|x-x_0|^2}{4|t_0+2\delta^2-t|} \right) G(x,t)      \\
         & \le C \exp \left(c\sigma^2 \frac{|x|^2}{R^2} \right)G(x,t) \\
         & \le \left\{ \begin{array}{ll}
         CG(x,t) & \textrm{if} |x| \le R/\sigma\\
         CR^{-n}\exp (-c\delta^2) & \textrm{if} |x| > R/\sigma
         \end{array} \right. \\
         & \le CG(x,t)+ CR^{-2}\exp \left((2-n)\log R-c\sigma^2 \right) \\
         & \le CG(x,t)+\epsilon R^{-2}
         \end{split}
         \end{equation}
 In the above inequalities we use the fact that $ |t-T|<4R^2 $ and $ (x_0,t_0)\in P_\rho \subset P_r $. With these inequalities we see that 
\[ 1 \le C \int_{P_1} e(\tilde{J})dxdt+\frac{C}{e_0^2} \le C\Psi(R) + C\epsilon E_0 + CR(E_0+1) + Cr^4 \]
and the constant $ C $ is only depending on $ M $. Now if $ \epsilon_0, \epsilon  $ and $ \sigma $ are small enough we obtain a contradiction.
So $ r_0 \le 1 $ and $ \max\limits_{0\le \a \le r} \{ (r-\a)^2\sup\limits_{P_\a}e(J) \} \le 4\rho_0^2 e_0 =4r_0 \le 4 $. Let $ \a=\frac{1}{2}r $ then we obtain the result.

\end{proof}

\subsection{Proof of main theorems}\label{sub33}

We start with two lemmas. 
Denote \[  \bar{e}(t)=\max\limits_M e(J(t))   \]

\begin{lem}\label{lemma2}
There exists a $ \delta >0 $ depending only on $ (M,g) $ such that for any $ t_0 \in [0,T) $, we have  \[  t_0+\delta \arctan \frac{1}{2\bar{e_0}} <T  \] and, if $ t_0 \le t \le t_0+\delta \arctan \frac{1}{2\bar{e_0}}  $, \[ \bar{e}(t) \le 2\bar{e}_0+\frac{1}{\bar{e_0}} \]
where $ \bar{e}_0=\bar{e}(t_0) $.
\end{lem}

\begin{proof}

 For any  $ x\in M $ with $ e(J)(x,t)=\bar{e}(t) $, we have $ \D e(J)(x,t) \le 0 $. So by \eqref{eJ1}, for all such $ x $ we have 
  \[ \p_t e(J)|_{(x,t)} \le C(\bar{e}(t)^2+1).  \]  
This implies \[ D^+ \bar{e}(t) \le  C(\bar{e}(t)^2+1) ,\] where \[  D^+f(t) = \limsup\limits_{h\to 0^+} \dfrac{f(t+h)-f(t)}{h}. \]
Let $ \delta=1/C $ and
 \[ g(t)= \dfrac{\bar{e}_0 +\tan C(t-t_0)}{1-\bar{e}_0 \tan C(t-t_0)},\quad  t\in [t_0, t_0+\delta \arctan \frac{1}{\bar{e}_0}  )  \] so that $ g(t_0)=\bar{e}_0 $ and  $ g'(t)=C(g(t)^2+1) $. By the comparison theorem of ODE, we have, for $ t_0 \le t \le  t_0+\delta \arctan \frac{1}{\bar{e}_0} $,  \[ \bar{e}(t) \le g(t). \] 
Note that on $  t_0 \le t \le t_0+\delta \arctan \frac{1}{2\bar{e_0}} $, \[ g(t) \le 2\bar{e}_0+\frac{1}{\bar{e_0}},  \]
the lemma follows. 

\end{proof}

\begin{cor}
Let $ (M,g,J_0) $ be an almost Hermitian structure such that $ |\n J_0| \le K $. Then the solution of \eqref{E3} with initial $ J_0 $ exists at least on    $ [0,\delta \arctan \frac{1}{2K^2}] $ .
\end{cor}

\begin{lem}\label{lemma3}
Let $ J $ be a harmonic almost complex structure on $ (M,g) $, i.e.
\[  \D J -J \n_p J \n_p J=0.  \]
There is a constant $ \epsilon _1 $, depending on $ (M,g) $, such that if \[ E(J) \le \epsilon_1 \] 
then $ J $ is K\"ahler.
\end{lem}

\begin{proof}
Suppose the lemma is not true. Then we may assume there is a sequence of non-K\"ahler  harmonic almost complex structure $ J_k $ on $ (M,g) $ such that \[ E(J_k) \to 0. \]
The $ \epsilon $-regularity implies that  $ | \n J_k | $ (hence $ |\n^m J_k| $ for any $ m $) is bounded for large $ k $. So we may find a subsequence, still denoted by $ J_k $ for convenience,  that converges to a harmonic almost complex structure $ J_0 $ with \[ E(J_0)=0 . \]
So that $ \n J_0=0 $  and 
\begin{equation}\label{k1}
\D(J_k-J_0) = J_k \n_p (J_k-J_0  )\n_p (J_k-J_0  ).
\end{equation} 
By \eqref{k1} we have 
\begin{displaymath}
\begin{split}
\int_M \left|  \n (J_k -J_0)  \right|^2 dv & = -\int_M \left(J_k -J_0, \D (J_k -J_0) \right) dv \\
& =-\int_M \left(   J_k -J_0,  J_k \n_p (J_k-J_0  )\n_p (J_k-J_0  )  \right)dv \\
& \le C \sup |J_k -J_0 | \int_M \left|  \n (J_k -J_0)  \right|^2 dv\\
&<\int_M \left|  \n (J_k -J_0)  \right|^2 dv,
\end{split}
\end{displaymath}
for $k$ sufficiently large and hence we have, for such $k$, 
\begin{equation}
\int_M \left|  \n (J_k -J_0)  \right|^2 dv=0.
\end{equation}  
 Hence $ \n J_k = \n J_0 =0 $. This is a contradiction since $ J_k $ is non-K\"ahler.

\end{proof}

Now we can prove Theorem \ref{main1} and Theorem \ref{main2}.

\begin{proof}[Proof of Theorem \ref{main1}]

We only need to show $ \sup \{ \bar{e}(t)| t\in [0,T) \} <+\infty $. If so, by lemma \ref{lemma2} we have $ T=+\infty $. And by Shi-type estimate $ |\n^m J| $ are bounded for all $ m $, so that the flow converges to a $ J_\infty $. Since 
\[ \dfrac{d E(J)}{dt}= -\int_M | \D J -J \n_p J \n_p J|^2 dv, \]
we have 
\[ 0\le  E(J_0)-E(J_\infty) =\int_{0}^{\infty} \int_M | \D J -J \n_p J \n_p J|^2 dv  \le \epsilon .\]
This implies that 
\[ \liminf\limits_{t\to \infty} | \D J -J \n_p J \n_p J|^2 =0 , \]
i.e. 
\[ \D J_\infty -J_\infty \n_p J_\infty \n_p J_\infty =0 . \]
Moreover, we have $ E(J_\infty) \le E(J_0) \le \epsilon $. Then lemma \ref{lemma3} shows that   $ (M,g,J_\infty) $ is  K\"ahler if  $ \epsilon $ is small enough.\\

Now we assume  that $ \sup \{ \bar{e}(t)| t\in [0,T) \} =+\infty $ and prove by contradiction. Then there is a sequence $ t_i \to T $ such that \[ \bar{e}(t_i) \to +\infty, \]
  \[ \lambda_i^2 \equiv \arctan \frac{1}{2 \bar{e}(t_i) } \to 0. \]
Let $ p_i \in M $ be such that \[ e(J)(p_i,t_i)= \bar{e}(t_i)  \] and $ \{x_\alpha \} $ be the normal coordinates centre at $ p_i$. In this coordinate system we may define $ Z(t) $ as in \eqref{Zt}, with $ T_0=t_i+\delta \lambda_i ^2  $, note that by lemma \ref{lemma2}  $ T_0<T $. 
On $Q_i \equiv B_{\lambda_i^{-1}}  \times [-\lambda_i^{-2} t_i,\delta] $ we define \[ \tilde{J}(x,t)=J(\lambda_i x, t_i+\lambda_i ^2 t ). \] 
Then $ \tilde{J} $ satisfies the equation 
\begin{equation}
\p_t \tilde{J} =\widetilde{\D} \tilde{J} -\tilde{J} \widetilde{\n}_p \tilde{J} \widetilde{\n}_p \tilde{J}
\end{equation}
with  $\widetilde{\D}  $ and $\widetilde{\n}$ respect to the metric $ \tilde{g}_{\alpha\beta}(x)=g_{\alpha\beta}(\lambda_i x) $. If $ i $ is large enough, we see that 
\[  e(\tilde{J}) (0,0) = \lambda_i^2 \bar{e}(t_i) > \frac{1}{4}\]
and, by lemma 3.3, 
\[    e(\tilde{J}) (x,t) \le \lambda_i^2 \bar{e}(t_i+\lambda_i^2 t) \le \lambda_i^2 \left(2\bar{e}(t_i)+\frac{1}{\bar{e}(t_i)} \right) <2 \]
since $ \lambda_i^2 \bar{e}(t_i) \to \frac{1}{2} $.
From \eqref{eJ} we have 
\[  (\p_t -\tilde{\D} )e(\tilde{J}) =\lambda_i^4 (\p_t -\D )e(J) \le C(e(\tilde{J}) +\lambda_i^4)  \]
i.e. $ \p_t h\le \tilde{\D} h   $ for $ h \equiv (e(\tilde{J})+ \lambda_i^4)\exp (-Ct) $.
Let $ Q=B_1 \times \left(  -\min\{ \frac{\delta}{2},\frac{\delta}{C} \} , \frac{\delta}{2}  \right) $ be a subset of $ Q_i $. By Moser's Harnack inequality (see \cite{Moser}, Theorem 3), there is a constant $ \gamma >0$ such that 
\[  \frac{1}{4}  < h(0,0) \le \gamma \left(  \dfrac{2}{\delta \text{Vol}(B_1)} \int_Q h^2dxdt \right)^{\frac{1}{2}} .\] 
Since $ e(\tilde{J})<2  $ and $ \exp (-Ct)  \le \exp (\delta)$ in $ Q $, we have 
\[  1 \le 16 \gamma^2 \left(    \dfrac{2(2+\lambda_i^4)\exp(2\delta)}{\delta \text{Vol}(B_1)} \int_Q (e(\tilde{J})+\lambda_i^4  )dxdt   \right) . \]
Moreover, we may assume $ \sqrt{\tilde{g}} > \frac{1}{2} $, $ \lambda_i <1  $ and $   16\gamma^2 \cdot {2(2+\lambda_i^4)\exp(2\delta)} \cdot \lambda_i^4  < \frac{1}{2}  $ for sufficiently large $ i $, so that 
\begin{equation}\label{moser}
1 \le \gamma_1 \int_Q | \widetilde{\n} \tilde{J} |^2 d\tilde{v}dt
\end{equation}
with constant $ \gamma_1 = \dfrac{96\gamma^2\exp (2\delta)}{\delta \text{Vol}(B_1)} $.

Consider the function \[  Z(t)=(T_0-t) \int_{\R^n} |\n J|^2 G \phi^2 \sqrt{|g|} dx  \]
on $ T_0-\rho < t < T_0 $, where $ \rho = \min \{ 1,T_0 \} $. By the monotonicity formula \eqref{z1}, we have 
\begin{equation}\label{z3}
 Z(t)\le e^{-Cf(T_0-\rho)}Z(T_0-\rho)+C\left( N^{n/2} (E_0+\sqrt{E_0})+\frac{1}{\ln ^2 N} \right)\rho
\end{equation}
for any $ N>1 $. Since 
\[ Z(T_0-\rho) \le C \rho^{1-n/2} \int_{B_1} | \n J |^2(x,T_0-\rho) dv \le C\rho^{1-n/2} E(J_0) \le C\epsilon \rho^{1-n/2}, \]
and $ -f(T_0-\rho) $ is bounded with $ \rho \le 1 $, \eqref{z3} implies that 
\[  Z(t) \le CN^{n/2} \rho^{1-n/2} \sqrt{\epsilon} + \frac{C}{\ln ^2 N}, \]
where the constant  $ C $ does \textbf{not} depend on $ N $.\\ 

Now on $ -\min \{ \frac{\delta}{2}, \frac{\delta}{C} \}<t< \frac{\delta}{2} $, we have 
\[  \frac{1}{2} \delta \lambda_i^2 < T_0-(t_i+\lambda_i^2 t)=\lambda_i^2 (\delta -t) <  \frac{3}{2} \delta \lambda_i^2 . \]
So that  \[ G_{(0,T_0)} (x,t_i+\lambda_i^2 t) = \dfrac{\exp (- |x|^2 / 4( T_0-t_i-\lambda_i^2 t) )}{\left(  4\pi( T_0-t_i-\lambda_i^2 t )  \right)^{n/2}} \ge C\delta^{-n/2}\lambda^{-n} \exp(-\frac{1}{2\delta}) \] provided $ |x|\le \lambda_i $.
Then we can estimate 
\begin{equation}\label{z4}
\begin{split}
\int_{B_1} | \widetilde{\n} \tilde{J} |^2 d \tilde{v} &= \lambda_i^{2-n} \int_{B_{\lambda_i}} | \n J |^2 (x,t_i+\lambda_i^2 t)dv\\
& \le C\delta^{\frac{n-2}{2}}\exp( \frac{1}{2\delta} )Z(t_i+\lambda_i^2 t)\\
& \le  CN^{n/2} \rho^{1-n/2} \sqrt{\epsilon} + \frac{C}{\ln ^2 N}
\end{split}
\end{equation}
By \eqref{moser} and \eqref{z4}, there is a constant $ C $ depending only on $ (M,g) $ such that  
\[  1\le  CN^{n/2} \rho^{1-n/2} \sqrt{\epsilon} + \frac{C}{\ln ^2 N}   .  \]
Choose $ N=\exp (\sqrt{2C}) $, so that we have 
\begin{equation}\label{rho}
 \rho^{n/2-1} \le 2C \exp ( n\sqrt{2C}/2 )  \sqrt{\epsilon} \equiv C_1 \sqrt{\epsilon} .
\end{equation}  
So if $ \epsilon $ is small enough,  there must have $ \rho=T_0 $. And since $ T_0 \to T  $ as $ i \to \infty $, \eqref{rho} implies that
\begin{equation}\label{T}
T^\frac{n-2}{2} \le C_1 \sqrt{\epsilon} ,
\end{equation} 
which contradicts to Collary 3.1 provided $ \epsilon=\epsilon(K) $ small enough.

\end{proof}

\begin{proof}[Proof of Theorem \ref{main2}]
Now suppose $ J $ with energy $ E(J)<\epsilon $ is in a homotopy class $ [J_0] $ that no K\"ahler structure exists. According to the proof above, we must have $ \sup \{\bar{e}(t)| t\in [0,T]) \} =+\infty $. Otherwise the flow will exist for all time and converge by subsequence to a K\"ahler structure. Since the convergence is in smooth topology, the limit K\"ahler structure will be in the same homotopy class, which is a contradiction. Hence \eqref{T} holds and  it implies that $T \to 0 $ as $ \epsilon \to 0 $.
\end{proof}

  Now we construct an example such that  the infimum of the energy functional restricted in the homotopy class of an almost complex structure is zero, but there is no compatible K\"ahlerian complex structure in the homotopy class. By Donaldson \cite{Don} (Corollary 6.5), on the K3 surface $ (M, g) $ there exists a homotopy class of almost complex structures that contains no complex structures. 
Donaldson's result on $K3$ surface has been generalized greatly in \cite{FVO}. 
  Inspired by these examples,  we consider a flat four-tori $T^4=S^1\times S^1\times S^1\times S^1$ with the flat metric $g_0$. Let $J_0$ be a standard complex structure on $T^4$. 
  Since the twistor bundle of a flat tori is $T^4\times SO(4)/U(2)=T^4\times S^2$, hence
  an almost complex structure on $T^4$ which is compatible with $g_0$ can be thought as a map from $T^4$ to $S^2$
  while $J_0$ corresponds to a constant map from $T^4$ to $S^2$.
Take a small ball $B_{r_0}(p)\subset T^4$ for some $r_0>0$, we can construct a smooth  almost complex structure $J$ such that $J$ agrees with $J_0$ outside $B_{r_0}$, but the homotopy class $[J]$ differs $[J_0]$ by a nonzero element in $H^4(T^4, \pi_4(SO(4)/U(2)))=\pi_4(S^2)=\mathbb{Z}_2$, see S. Donaldson \cite{Don} for such a construction on $K3$ surface. 
We can think $J_0$ is a constant map from $T^4$ to $S^2$. Since $J$ agrees with $J_0$ outside $B_{r_0}$, when restricted in $B_{r_0}$, we can think $J$ as a map from $B_{r_0}$ to $S^2$ with the boundary $\p B_{r_0}$ contracted to a point, hence $J$ (restricted on $B_{r_0}$) can be viewed as a map from $S^4$ to $S^2$. The homotopy class of $J$ then corresponds to the nonzero element in $\pi_4(S^2)=\mathbb{Z}_2$. 
 
 Given any such a smooth almost complex structure $J$, we construct an almost complex structure $J_r$ for $r\in (0, r_0]$ such that $J_r=J_0$ outside the ball $B_r(p)$ and
 \[
 J_{r}(x)=J(xr_0/r), x\in B_r(p). 
 \]
 Clearly $J_r$ is in the homotopy class of $[J]$. We compute the energy of $J_r$ by
 \[
 E(J_r)=\int_{B_r} |\nabla J_r|^2 dx=\int_{B_{r}} |\nabla (J(xr_0/r))|^2dx=r_0^{-2} r^2\int_{B_{r_0}} |\nabla J(y)|^2 dy
 \]
 Hence $E(J_r)$ goes to zero when $r\rightarrow 0$. On $(T^4, g_0)$, the harmonic heat flow with an initial almost complex structure $J_r$, for $r$ sufficiently small, must blow up at finite time by Theorem \ref{main2}. Otherwise, if the flow exists for all time, then by $\epsilon$ regularity we know that $J(t)$ converges smoothly to $J_\infty$, which is compatible with $g_0$ and defines a K\"ahler structure, and $J_\infty$ lies in the homotopy class $[J]$. But $\nabla J_\infty=0$ implies that the homotopy class of $J_\infty$ corresponds to a constant map from $T^4$ to $S^2$. This is a contradiction.

 Similar examples can be constructed for Donaldson's $K3$ example and examples considered in \cite{FVO} with necessary modifications. For simplicity we omit the details.

%There are interesting work recently to study various geometric evolution equations on Hermitian manifolds, symplectic manifolds, and in mostly general case, on almost Hermitian manifolds, we refer to \cite{LW, G, TW, ST1, ST2, TWY, He16} and reference therein. 

\end{document}